\documentclass[a4paper,12pt]{amsart}

\usepackage[utf8]{inputenc}
\usepackage[T1]{fontenc}
\usepackage{lmodern}
\usepackage[a4paper]{geometry}
\usepackage[english]{babel}

\usepackage{enumerate}
\usepackage[colorlinks=true]{hyperref}
\hypersetup{linkcolor=blue, citecolor=red}

\usepackage{amsmath, amsthm, amssymb}

\theoremstyle{plain}
\newtheorem{thm}[subsection]{Theorem}
\newtheorem{prop}[subsection]{Proposition}
\newtheorem{cor}[subsection]{Corollary}
\newtheorem{que}[subsection]{Question}
\newtheorem{lem}[subsection]{Lemma}

\theoremstyle{definition}

\theoremstyle{remark}
\newtheorem{rem}[subsection]{Remark}

\newcommand{\A}{\ensuremath{\mathcal{A}}}

\newcommand{\D}{\ensuremath{\mathcal{D}}}
\newcommand{\E}{\ensuremath{\mathcal{E}}}
\newcommand{\m}{\ensuremath{\mathfrak{m}}}
\renewcommand{\O}{\ensuremath{\mathcal{O}}}
\newcommand{\bbQ}{\ensuremath{\mathbb{Q}}}
\newcommand{\G}{\ensuremath{\mathcal{G}}}
\newcommand{\bbG}{\ensuremath{\mathbb{G}}}

\newcommand{\J}{\ensuremath{\mathcal{J}}}
\newcommand{\W}{\ensuremath{\mathcal{W}}}
\newcommand{\F}{\ensuremath{\mathcal{F}}}

\newcommand{\bbZ}{\ensuremath{\mathbb{Z}}}
\newcommand{\bbR}{\ensuremath{\mathbb{R}}}
\newcommand{\bbN}{\ensuremath{\mathbb{N}}}
\renewcommand{\phi}{\ensuremath{\varphi}}
\renewcommand{\r}{\ensuremath{\rightarrow}}
\newcommand{\ab}{\ensuremath{\mathrm{ab}}}
\renewcommand{\dim}{\ensuremath{\mathrm{dim}}}
\renewcommand{\mod}{\ensuremath{\mathrm{mod}}}
\newcommand{\tor}{\ensuremath{\mathrm{tor}}}
\newcommand{\tors}{\ensuremath{\mathrm{tors}}}
\newcommand{\ord}{\ensuremath{\mathrm{ord}}}

\DeclareMathOperator{\Res}{\mathrm{Res}}
\DeclareMathOperator{\Gr}{\mathrm{Gr}}

\DeclareMathOperator{\Gal}{\mathrm{Gal}}
\DeclareMathOperator{\Jac}{\mathrm{Jac}}
\DeclareMathOperator{\Hom}{\mathrm{Hom}}
\DeclareMathOperator{\Norm}{\mathrm{Norm}}

\title{Splitting properties of the reduction of semi-abelian varieties}

\author{Alan Hertgen}

\address{Universit\'e de Bordeaux, Institut de Math\'ematiques de Bordeaux, UMR 5251, 351 cours de la lib\'eration, 33405 Talence cedex, France}

\email{alan.hertgen@math.u-bordeaux.fr}

\subjclass[2010]{11G07, 11G10, 14G20, 14G40}

\begin{document}

\begin{abstract}
Let $K$ be a complete discrete valuation field. Let $\O_K$ be its ring of integers. Let $k$ be its residue field which we assume to be algebraically closed of characteristic exponent $p\geq1$. Let $G/K$ be a semi-abelian variety. Let $\G/\O_K$ be its N\'eron model. The special fiber $\G_k/k$ is an extension of the identity component $\G_k^0/k$ by the group of components $\Phi(G)$. We say that $G/K$ has \emph{split reduction} if this extension is split.

Whereas $G/K$ has always split reduction if $p=1$ we prove that it is no longer the case if $p>1$ even if $G/K$ is tamely ramified. If $J/K$ is the Jacobian variety of a smooth proper and geometrically connected curve $C/K$ of genus $g$, we prove that for any tamely ramified extension $M/K$ of degree greater than a constant, depending on $g$ only, $J_M/M$ has split reduction. This answers some questions of Liu and Lorenzini.
\end{abstract}

\maketitle

\section*{Introduction}

Let $K$ be a complete discrete valuation field. Let $\pi_K$ be a uniformizing element of $K$. Let $v_K$ be the discrete valuation on $K$ normalized such that $v_K(\pi_K)=1$. Let $\O_K$ be the ring of integers. Let $k$ be the residue field which we assume to be algebraically closed of characteristic exponent $p\geq1$. Let $G/K$ be a semi-abelian variety with N\'eron (lft) model $\G/\O_K$ (see \cite[Chapter 10]{BLR}). Let $\G_k=\G\times_{ \O_K} k$ be the special fiber of $\G/\O_K$. We have an exact sequence \begin{equation}0 \r \G^0_k \r \G_k \r \Phi(G) \r 0 \label{es}\end{equation} where $\G^0_k/k$ is the identity component of $\G_k/k$ and $\Phi(G)$ is the group of components which is known to be a finitely generated abelian group. Following \cite[Introduction]{LL} we shall say that $G/K$ has \emph{split reduction} if this exact sequence is split. In other words, $G/K$ has split reduction when $\G_k/k$ is isomorphic to the direct product $\G_k^0\times_k\Phi(G)$ as an algebraic group. When $\Phi(G)$ is finite, $G/K$ has split reduction if and only if for each $\phi\in\Phi(G)$ there exists $x\in\G_k(k)$ lifting $\phi$ and with the same order. 

We know that a semi-abelian variety $G/K$ has split reduction when $p=1$ (see \cite[Proposition 1.4]{LL}), or when $\G^0_k/k$ is semi-abelian (see \cite[Proposition 1.6]{LL}), or when $G/K$ is a tamely ramified abelian variety (\emph{i.e.} the minimal finite separable extension such that $G/K$ acquires semi-abelian reduction is tamely ramified) with toric rank equal to $0$ (see \cite[Corollary 1.9]{LL}). This makes natural the question whether split reduction is automatic for tamely ramified semi-abelian varieties (this is \cite[Question 1.10]{LL}).

Let us also recall the notion of totally not split semi-abelian variety from \cite[\S 1.2]{LL}. We say that an exact sequence of finitely generated abelian groups \[0\r H^0 \r H \r \Psi \r 0\] is \emph{totally not split} (for a fixed $p$) if for each $\phi\in\Psi$ of order $p$ there is no $x\in H$ of order $p$ lifting $\phi$. We say that a semi-abelian variety $G/K$ has \emph{totally not split reduction} if the exact sequence (\ref{es}) is totally not split.

In \cite{LL}, Liu and Lorenzini studied in detail the case of elliptic curves and norm tori together with their duals. For such a semi-abelian variety $G/K$, they found that there exists a constant $c_1$ depending only on the dimension of $G/K$ such that if $G/K$ has totally not split reduction then the Swan conductor (see \cite[\S 2.1]{Se2} for the definition) of $G/K$ is positive and bounded by $c_1$. For some classes of tori they found that there exists a constant $c_2$ depending only on the dimension of $G/K$ and on the absolute ramification index $v_K(p)$ such that if the Swan conductor of $G/K$ is greater than $c_2$ then $G/K$ has split reduction. Finally, they found that there exists a constant $c_3$ depending only on the dimension of $G/K$ such that $G_M/M$ has split reduction for any tamely ramified extension $M/K$ of degree greater than $c_3$. They ask whether similar statements hold in greater generality in \cite[Questions 6.9, 6.10 and 6.11]{LL} respectively. 

The aim of this paper is to provide answers to these questions. We answer negatively \cite[Question 1.10]{LL} in \S\ref{ctrex} by constructing a family of tamely ramified abelian varieties which do not have split reduction. We answer negatively \cite[Questions 6.9]{LL} in \S \ref{no69} by constructing a family of simple abelian varieties which have totally not split reduction but whose Swan conductors cannot be bounded independently of the field of definition. We answer negatively \cite[Questions 6.10]{LL} in \S \ref{no610} by constructing an abelian variety whose Swan conductor is as large as possible but which does not have split reduction. Finally, we give a positive answer to \cite[Questions 6.11]{LL} with Corollary \ref{yes611} which states that Jacobian varieties acquire split reduction after sufficiently large tamely ramified extensions. Our counterexamples are Weil restrictions of elliptic curves, so we give general considerations about the splitting properties of Weil restrictions in Section \ref{III}.

\subsection*{Acknowledgments} I would like to thank my thesis advisor Qing Liu for his guidance, Dino Lorenzini for helpful comments and suggestions and the referee for a careful reading of the manuscript.

\section{A review of Liu and Lorenzini's results}\label{I}

\subsection{The general case}\label{S11} For an abelian group $H$, we will denote by $H_{\tors}$ the torsion subgroup of $H$ and by $H_p$ its $p$-primary part. It is shown in \cite[Proposition 1.4]{LL} that the sequence \[0 \r \G^0_{k,p} \r \G_{k,p} \r \Phi(G)_p \r 0\] is exact and that $G/K$ has (totally not) split reduction if and only if this  exact sequence is (totally not) split. So, the matter of splitting is only non-trivial for $p>1$, what we will assume from now on. Another consequence of this fact is that if $|\Phi(G)_{\tors}|$ is prime to $p$ then $G/K$ has split reduction (see \cite[Corollary 1.5]{LL}). Let us also mention that if $\G^0_k/k$ is a semi-abelian variety then $G/K$ has split reduction (see \cite[Proposition 1.6]{LL}).

\subsection{Tamely ramified semi-abelian varieties} Let $G/K$ be a semi-abelian variety. There exists a minimal finite separable extension $L/K$ such that $G_L=G\times_{ K} L$ has semi-abelian reduction (see  \cite[Theorem 2.3.6.5 (2)]{HN1}). Let $G_{\tor}/K$ and $G_{\ab}/K$ be the toric and the abelian parts of $G/K$. Note that $G_{\tor}\times_K L$ is a split torus (that is to say isomorphic to some power of $\bbG_{m,L}/L$). The following proposition is a combination of \cite[Corollaries 1.7 and 1.9]{LL}.

\begin{prop}\label{1}
Let $G/K$ be a semi-abelian variety whose abelian part $G_{\ab}/K$ has toric rank $0$. Let $L/K$ be the minimal finite separable extension such that $G_L/L$ has semi-abelian reduction. If $L/K$ is tamely ramified then $G/K$ has split reduction.
\end{prop}

\begin{proof}
We adapt the proof of \cite[Theorem 7.1.2.8]{HN1}. By \cite[Proposition 4.1.1.3]{HN1}, the sequence \begin{equation}\Phi(G_{\tor}) \r \Phi(G) \r \Phi(G_{\ab}) \r 0 \label{es2}\end{equation} is exact. Now, $\Phi(G_{\ab})$ is killed by $\left[L:K\right]^2$ by \cite[Proposition 1.8]{LL}. By \cite[Corollary 5.4]{HN2}, the kernel of the canonical morphism $\Phi(G_{\tor})\r\Phi(G_{\tor}\times_KL)$ is killed by $\left[L:K\right]$. As the group of components of a split torus is free, $\Phi(G_{\tor})_{\tors}$ is in the kernel of this morphism and hence is killed by $\left[L:K\right]$. Now, the exact sequence (\ref{es2}) implies that $\Phi(G)_{\tors}$ is killed by $\left[L:K\right]^3$ and in particular its order is prime to $p$. We have recalled in \S\ref{S11} that this implies that $G/K$ has split reduction.
\end{proof}

\begin{cor}\label{cor}
Let $E/K$ be an elliptic curve. Let $L/K$ be the minimal finite separable extension such that $E_L/L$ has semi-abelian reduction. If $L/K$ is tamely ramified then $E/K$ has split reduction.
\end{cor}

\begin{proof}
If $E/K$ has semi-abelian (multiplicative or good) reduction then this is a consequence of our last remark in \S\ref{S11}. Otherwise, the toric rank of $E/K$ is $0$ and this follows from Proposition \ref{1}. 
\end{proof}

\subsection{Two particular cases}\label{S15} Let us insist on two particular cases of Proposition \ref{1}. First, an abelian variety $A/K$ which has potentially good reduction over a tamely ramified extension $L/K$ has split reduction. Note that one can prove this using again \cite[Corollary 5.4]{HN2} and the fact that the group of components of an abelian variety with good reduction is trivial. Second, a torus $T/K$ which splits over a tamely ramified extension $L/K$ has split reduction. Indeed, $T_L/L$ is isomorphic to a split torus if and only if $T_L/L$ has semi-abelian reduction.

\subsection{Non-Archimedean uniformization} Let us recall that an abelian variety $A/K$ admits a non-Archimedean uniformization as follows (see \cite[Theorem 1.2]{BX}). There exist a semi-abelian variety $G/K$ and a lattice $\Lambda/K$ in $G/K$ such that the sequence of rigid analytic groups \[0 \r \Lambda \r G \r A \r 0 \] is exact and such that $G/K$ is an algebraic extension \[ 0 \r T \r G \r B \r 0\] of an abelian variety $B/K$ with potentially good reduction by a torus $T/K$. 

Let us denote by $\delta(G/K)$ the Swan conductor of $G/K$. Recall that $\delta(G/K)$ is zero if and only if $G/K$ acquires semi-abelian reduction after a tamely ramified extension $L/K$. Considering this non-Archimedean uniformization and \S\ref{S15} it is natural to ask the following question.

\begin{que}\cite[Question 1.10]{LL}\label{110}

Let $G/K$ be a semi-abelian variety. Let $L/K$ be the minimal finite separable extension such that $G_L/L$ has semi-abelian reduction. If $L/K$ is tamely ramified, is it true that $G/K$ has split reduction ? In other words, is it true that the Swan conductor $\delta(G/K)$ is positive if $G/K$ does not have split reduction ?
\end{que}

In spite of all this evidence, we will show that the answer to this question is no by constructing a family of counterexamples in \S\ref{ctrex}. Note that we need to consider abelian varieties over $K$ of dimension $>1$ (by Corollary \ref{cor}), positive toric rank (by Proposition \ref{1}) and which do not have semi-abelian reduction over $K$ (by \S\ref{S11}).

\subsection{The case of elliptic curves} The main results in \cite{LL} about elliptic curves are gathered in the following theorem.

\begin{thm}\cite[Theorem 2.1 and Proposition 3.3]{LL}\label{21}

Let $E/K$ be an elliptic curve and let $\delta(E/K)$ be its Swan conductor. 

\begin{enumerate}[(1)]
\item If $E/K$ has totally not split reduction, then \[1\leq\delta(E/K)\leq3.\] 
\item If $E/K$ has not split but not totally not split reduction, then $E$ is of type $\mathbf{I}_{2n}^{\ast}$ for some integer $n$ and \[1\leq\delta(E/K)\leq 2n+3.\] 
\item Let $M/K$ be a tamely ramified extension of degree $\geq4$. Then $E_M/M$ has split reduction.
\end{enumerate}
\end{thm}

\subsection{The case of quotient tori}\label{S110} Let $L/K$ be a finite separable extension. Consider the Weil restriction $\Res_{L/K}\bbG_{m,L}$ of $\bbG_{m,L}/L$ under the extension $L/K$ (see \S\ref{S21} for the definition and first properties of Weil restriction). Let $S/K$ be the \emph{quotient torus} $(\Res_{L/K}\bbG_{m,L})/\bbG_{m,K}$. Define the \emph{norm torus} $T/K$, also denoted by $\Res^1_{L/K}\bbG_{m,L}$ in the sequel, to be the kernel of the norm map \[\Res_{L/K}\bbG_{m,L}\stackrel{\Norm_{L/K}}{\longrightarrow}\bbG_{m,K}.\] By \cite[Lemma 4.1]{LL}, $S/K$ is isomorphic to the dual torus of $T/K$. Moreover, if $L/K$ is a cyclic extension then $S/K$ is isomorphic to $T/K$. The main results about tori in \cite{LL} are about this kind of tori and are gathered in the following theorem.

\begin{thm}\cite[Theorem 4.6 and Corollary 4.11]{LL}\label{46}

Let $S/K$ be the quotient torus $(\Res_{L/K}\bbG_{m,L})/\bbG_{m,K}$ and let $\delta(S/K)$ be its Swan conductor. 
\begin{enumerate}[(1)]
\item The torus $S/K$ has totally not split reduction if and only if \[1\leq\delta(S/K)\leq\dim(S).\]
\item If $\delta(S/K)\geq(\dim(S)+1)\ord_p(\dim(S)+1)v_K(p)$, then $S/K$ has split reduction.
\item Let $M/K$ be a tamely ramified extension of degree $\geq\dim(S)+1$. Then, $S_M/M$ has split reduction.
\end{enumerate}
\end{thm}

\subsection{Some further questions} The results recalled in the last two theorems lead naturally to the following questions on possible generalizations.

\begin{que}\cite[Question 6.9]{LL}\label{69}

Let $g$ be a positive integer and consider all abelian varieties $A$ of dimension $g$ over a discrete valuation field $K$ with toric rank equal to $0$. Is there a constant $c_1$ depending on $g$ but not on the field $K$ such that if $A/K$ has totally not split reduction then the Swan conductor $\delta(A/K)$ is bounded by $c_1$ ?
\end{que}

This question has a positive answer for elliptic curves (see Theorem \ref{21} (1)) or for abelian varieties uniformized by quotient tori as above (see \cite[Theorem 6.6]{LL} which relies on Theorem \ref{46} (1)). As mentioned in \cite[Question 6.9]{LL} one can construct obvious counterexamples by taking the product of an abelian variety with totally not split reduction by an elliptic curve with trivial group of components and large Swan conductor. We will construct in \S \ref{no69} a family of simple abelian varieties which have totally not split reduction and whose Swan conductors really depend on the field of definition so that the answer is no even for simple abelian varieties.

\begin{que}\cite[Question 6.10]{LL}\label{610}

Assume that $K$ is of characteristic $0$. Let $A/K$ be an abelian variety of dimension $g$. The Swan conductor of $A/K$ is bounded by a constant depending on $g$ and the absolure ramification index $v_K(p)$ only by \cite[Proposition 6.2 and Proposition 6.11]{BK}. Is there a bound $c_2$ depending on $g$ and $v_K(p)$ only such that if $\delta(A/K)>c_2$ then $A/K$ has split reduction ?
\end{que}

Of course, we want $c_2$ to be smaller than the absolute bound of \cite{BK}. Such a bound exists for quotient tori (see Theorem \ref{46} (2)) but we will show that the answer is no in general by giving in \S \ref{no610} an example of abelian variety whose Swan conductor achieves the bound from \cite{BK} but which does not have split reduction. Let us however mention the following result which implies the existence of such a bound for elliptic curves.

\begin{prop}
Assume that $K$ is of characteristic $0$. Let $E/K$ be an elliptic curve. If $\delta(E/K)$ achieves the bound from \cite{BK}, then $E/K$ has split reduction.
\end{prop}

\begin{proof}
If $p\neq2,3$, then $E/K$ is tamely ramified and hence has split reduction by Corollary \ref{cor}. If $p=3$, then the bound of \cite{BK} is $3v_K(3)\geq3$ (see also \cite{LRS}). Only elliptic curves with reduction type $\mathbf{III}$ or $\mathbf{III}^{\ast}$ may not have split reduction and in this case we have $\delta(E/K)=1$ or $2$ by \cite[\S 2.13 and \S 2.14]{LL}, whence the result.

Let $p=2$. In this case, the bound of \cite{BK} is $6v_K(2)>3$ (see also \cite{LRS}). Therefore by Theorem \ref{21} (1) an elliptic curve whose conductor achieves the bound of \cite{BK} does not have totally not split reduction. Assume that $E/K$ does not have split reduction and that $\delta(E/K)=6v_K(2)$. By Theorem \ref{21} (b) we know that $E/K$ has reduction type $\mathbf{I}_{2n}^{\ast}$ for some integer $n$. Now, we will follow the proof of \cite[Proposition 2.11 (b)]{LL}. Let $\Delta$ be the minimal discriminant of $E/K$. Then \cite[Proposition 2.11 (b)]{LL} implies that $v_K(\Delta)\leq 4n+9$. Assume $E/K$ is given by a minimal Weierstrass equation \[y^2+a_1xy+a_3x=x^3+a_2x^2+a_4x+a_6,\] with $a_i\in\O_L$ for $i\in\left\{1,2,3,4,6\right\}$. As in \cite{LL}, let us set $e=v_K(2)$ and $v=v_K(a_1)$. By Ogg's formula we have \[\begin{array}{lll}v_K(\Delta)&=&2+\delta(E/K)+(2n+5-1)\\&=&6e+2n+6.\end{array}\] 

Consider first the case where $e<v$. We have $v_K(\Delta)>4e+2n+6$. According to the table of valuations in \cite{LL}, $4e+2n+6$ has to be greater or equal to $4n+8$ or $3e+3n+7$. If $4e+2n+6\geq 4n+8$ then $2e\geq n+1$ and $v_K(\Delta)\geq5n+9$ which is false. If $4e+2n+6\geq 3e+3n+7$ then $e\geq n+1$ and $v_K(\Delta)\geq 8n+12$ which is false.

Consider the case where $e\geq v$. If $v=1$, then $v_K(\Delta)=2n+8$ but this is false so that we can assume $v>1$. We have $v_K(\Delta)>4v+2n+4$. According to the table of valuations in \cite{LL}, $4v+2n+4$ has to be greater or equal to $4n+8$ or $3v+3n+6$. If $4v+2n+4\geq 4n+8$ then $2v\geq n+2$ and $v_K(\Delta)\geq5n+12$ which is false. If $4v+2n+4\geq 3v+3n+6$ then $v\geq n+2$ and $v_K(\Delta)\geq 8n+18$ which is false. Hence, there is a contradiction, i.e. $\delta(E/K)=6v_K(2)$ implies that $E/K$ has split reduction.
\end{proof}

\begin{que}\cite[Question 6.11]{LL}\label{611}

Let $G/K$ be a semi-abelian variety of dimension $g$ that does not have split reduction.

\begin{enumerate}[(1)]
\item Is it always possible to find a tamely ramified extension $M/K$ such that $G_M/M$ has split reduction ?
\item Does there exist a constant $c_3$ depending on $g$ only such that if $M/K$ is any tamely ramified extension of degree greater than $c_3$, then $G_M/M$ has split reduction ?
\end{enumerate}
\end{que}

These questions have positive answers for elliptic curves or quotient tori (Theorem \ref{21} (3) and Theorem \ref{46} (3)). We will show that the answer to these questions is true for Jacobian varieties over $K$ in Corollary \ref{yes611}.

\section{Splitting properties of tamely ramified semi-abelian varieties}\label{II}

\subsection{Weil restriction}\label{S21} We recall below standard facts about Weil restrictions and we refer the reader to \cite[Section 7.6]{BLR} for details. Let $S'\r S$ be a morphism of schemes. Let $X'$ be a scheme over $S'$. The \emph{Weil restriction $\Res_{S'/S}X'$ of $X'$ under the morphism $S'\r S$}, when it exists, is the scheme over $S$ representing the functor on schemes over $S$ defined by \[T\mapsto \Hom_{S'}(T\times_S S',X').\] When $X'/S'$ is quasi-projective and $S'\r S$ is finite and locally free then $\Res_{S'/S}X'$ always exists. The notion of Weil restriction commutes with base change in the following sense. If $T\r S$ is a morphism of base change and if we write $T'=T\times_S S'$ then for any scheme $X'/S'$ there is a canonical isomorphism \begin{equation}\Res_{T'/T}(X'\times_{S'}T')\stackrel{\simeq}{\r}(\Res_{S'/S}X')\times_S T.\label{bc}\end{equation}

\subsection{Weil restriction and N\'eron models}\label{S22} Let $L/K$ be a finite separable extension of degree $d$. Let $G/L$ be a semi-abelian variety of dimension $g$. Then $\Res_{L/K}G$ is a semi-abelian variety over $K$ of dimension $d\cdot g$. Let $\O_L$ be the ring of integers of $L$. Let $\G/\O_L$ be the N\'eron model of $G/L$ and let $\G^0/\O_L$ be its identity component. Then, $\Res_{O_L/\O_K}\G$ is the N\'eron model of $\Res_{L/K}G$ over $\O_K$ (\cite[Proposition 7.6.6]{BLR}). The following proposition is inspired by \cite[Lemma 3.1]{NX}.

\begin{prop}\label{cpnt}
The identity component of $\Res_{\O_L/\O_K}\G$ is $\Res_{\O_L/\O_K}\G^0$ and we have the following isomorphism \[\Phi(\Res_{L/K}G)\stackrel{\simeq}{\r}\Phi(G).\]
\end{prop}

\begin{proof}
It follows from \cite[Proposition A.5.9]{CGP} that the fibers of $\Res_{\O_L/\O_K}\G^0$ are connected. From the immersion \[\G^0 \r \G\] we get a morphism \[\Res_{\O_L/\O_K}\G^0 \r \Res_{\O_L/\O_K}\G\] which factorizes through \[\Res_{\O_L/\O_K}\G^0 \r (\Res_{\O_L/\O_K}\G)^0.\] Now, since the Weil restriction commutes with open and closed immersions by \cite[Proposition 7.6.2]{BLR}, we have \[\Res_{\O_L/\O_K}\G^0\stackrel{\simeq}{\r}(\Res_{\O_L/\O_K}\G)^0.\] Finally, we have \[\begin{array}{lll}\Phi(\Res_{L/K}G)&\stackrel{\simeq}{\leftarrow}&(\Res_{\O_L/\O_K}\G)(\O_K)/(\Res_{\O_L/\O_K}\G^0)(\O_K)\\&\stackrel{\simeq}{\r}&
\G(\O_L)/\G^0(\O_L)\\&\stackrel{\simeq}{\r}&\Phi(G).\end{array}\]
\end{proof}

\subsection{Weil restriction of Tate curves} Let $L/K$ be a finite separable extension of degree $d\geq2$. Let $\pi_L$ be a uniformizing element of $L$. Let $v_L$ be the discrete valuation on $L$ normalized such that $v_L(\pi_L)=1$. Let $q\in L^{\times}$ be such that $v_L(q)>0$ and let $E/L$ be the Tate curve associated to $q$ (see \cite[Theorem V.3.1]{Si2}). Set $n=v_L(q)$. Let $\E/\O_L$ be the N\'eron model of $E/L$ and let $\E^0/\O_L$ be its identity component. Let us recall that we have the following isomorphisms (see \cite[Section V.4]{Si2})  \[\begin{array}{rcl}E(L)&\cong& L^{\times}/q^{\bbZ},\\ \E^0(\O_L)&\cong&\O_L^{\times}, \\ \Phi(E)&\cong&\bbZ/n\bbZ.\end{array}\]

Let us consider the abelian variety $A=\Res_{L/K} E/K$ obtained by Weil restriction under the extension $L/K$. In particular, $\dim(A)=d$. The N\'eron model of $A$ over $\O_K$ is $\A=\Res_{\O_L/\O_K}\E$ (see \S\ref{S22}) and its identity component is $\A^0=\Res_{\O_L/\O_K}\E^0$ (by Proposition \ref{cpnt}). The special fiber $\E_k/k$ is isomorphic to $\bbG_{m,k}/k$ (see \cite[Theorem V.5.3]{Si2}), thus we have by the base change formula (\ref{bc}) \[\A_k^0\cong\Res_{(\O_L/\pi_K\O_L)/k}\bbG_{m,(\O_L/\pi_K\O_L)}\] and then \[\A_k^0(k)\cong(\O_L/\pi_K\O_L)^{\times}.\] 

As $L/K$ is totally ramified, $\O_L$ is generated by $\pi_L$ over $\O_K$ and $\pi_L$ satisfies an Eisenstein equation \[t^d+a_1t^{d-1}+\cdots+a_d=0\] with $a_i\in\pi_K\O_K$ and $v_K(a_d)=1$. As in the proof of \cite[Proposition 3.2]{NX}, we have a split exact sequence \begin{equation}0 \r 1+\pi_L\O_L/\pi_K\O_L \r (\O_L/\pi_K\O_L)^{\times} \r k^{\times} \r 0.\label{untor}\end{equation} The group $k^{\times}$ is the group of closed points of a one-dimensional torus over $k$ and the group $1+\pi_L\O_L/\pi_K\O_L$ is the group of closed point of a unipotent algebraic group over $k$. Indeed, one has the composition serie \[\left\{1\right\}=1+\pi_L^d\O_L/\pi_K\O_L\subseteq1+\pi_L^{d-1}\O_L/\pi_K\O_L\subseteq\cdots
\subseteq1+\pi_L\O_L/\pi_K\O_L\] whose succesive quotients are isomorphic to $k$.  We may note that the toric rank of $A/K$ is positive and that it does not have semi-abelian reduction so that we are in the required situation to deal with Question \ref{110}.

We can describe the reduction map $A(K)\r\A_k(k)$ (which is surjective because $K$ is assumed to be complete). Let $P\in A(K)\cong E(L)\cong L^{\times}/q^{\bbZ}$. The image of $P$ in $\Phi(A)\cong\Phi(E)\cong\bbZ/n\bbZ$ is given by the class of $v_L(z)$ modulo $n$ where  $z\in L^{\times}$ is a preimage of $P$. If $P\in\A^0(\O_K)\cong\E^0(\O_L)\cong\O_L^{\times}$ then its image in $\A_k^0(k)\cong(\O_L/\pi_K\O_L)^{\times}$ is given by reduction modulo $\pi_K\O_L$. In particular, the kernel of the reduction map is isomorphic to $1+\pi_K\O_L$.

The next lemma is the analogue in our situation of \cite[Claim 4.7]{LL} which is about quotient tori.

\begin{lem}\label{order}
Let $m\in\bbN$ be a divisor of $n$. There exists a point in $\A_k(k)$ of order $m$ whose image in $\Phi(A)$ is also of order $m$ if and only if $\pi_L^n/q$ is a $m$-th power in $(\O_L/\pi_K\O_L)^{\times}$.
\end{lem}

\begin{proof}
The existence of such a point is equivalent to the existence of $z\in L^{\times}$ such that \[v_L(z)\equiv n/m \ \mod \ n\] and \[z^m\in q^{\bbZ}(1+\pi_K\O_L).\] Suppose that such $z$ exists. Multiplying $z$ by a suitable power of $q$ we may suppose that $v_L(z)=n/m$. Thus, necessarily $z^m/q\in 1+\pi_K\O_L$ (recall that $v_L(q)=n$). Then \[\pi_L^n/q=(\pi_L^{n/m}/z)^m(z^m/q)\in\O_L^m(1+\pi_K\O_L).\] Conversely, if $\pi_L^n/q\equiv y^m \ \mod \ \pi_K\O_L$ for some $y\in\O_L^{\times}$, then we can take $z=y\pi_L^{n/m}$.
\end{proof}

\begin{cor}\label{resE}
Assume that $v_L(q)=p$. Then, $A=\Res_{L/K} E/K$ has split reduction if and only if \begin{equation}q/\pi_L^p\in \O_K^{\times}((1+\pi_L\O_L)^p+\pi_K\O_L).\label{congr}\end{equation}
\end{cor}

\begin{proof}
This follows from the lemma and the exact sequence (\ref{untor}), using that the multiplication by $p$ is surjective on $k^{\times}\cong(\O_K/\pi_K\O_K)^{\times}$.
\end{proof}

\begin{rem}
In some sense, both the conditions of having split reduction and not having split reduction for $A/K$ as in Corollary \ref{resE} are open for the topology on $L$. Indeed, if we consider $q'$ close enough to $q$, i.e. such that $q'\in q(1+\pi_K\O_L)$, then $q$ satisfies condition (\ref{congr}) if and only if $q'$ satisfies condition (\ref{congr}).
\end{rem}

\begin{prop}\label{mult}
Let $L/K$ be a Galois extension. Then $A_L/L$ has purely multiplicative reduction.
\end{prop}

\begin{proof} Recall that $E/L$ has multiplicative reduction. Hence, the action of the absolute Galois group $\Gamma_L=\Gal(K^s/L)$ on the Tate module $T_{\ell}(E)$, $\ell\neq p$, is unipotent and non-trivial. Now, we have an isomorphism \[A_L\cong\prod_{\sigma\in \Gal(L/K)}{}^{\sigma}E.\] Let $\sigma\in \Gal(L/K)$. The action of $\Gamma_L$ on $T_{\ell}({}^{\sigma}E)$ is conjugated to its action on $T_{\ell}(E)$ and thus is unipotent and non-trivial. This implies that ${}^{\sigma}E/L$ has multiplicative reduction which proves the proposition.
\end{proof}

\subsection{Counterexample to Question \ref{110}}\label{ctrex} For any tamely ramified (hence Galois) extension $L/K$, one can choose $q\in L^{\times}$ with $v_L(q)=p$ such that condition (\ref{congr}) is not satisfied (for instance take $q=\pi_L^p(1+\pi_L)$). This way, we get tamely ramified abelian varieties (by Proposition \ref{mult}) which do not have split reduction (by Corollary \ref{resE}) and this gives a negative answer to Question \ref{110}.

\begin{rem}
Using Lemma \ref{order}, one can construct abelian varieties with a specified subgroup of $\Phi(A)$ lifting into $\A_k(k)$. More precisely, let us fix two non-negative integers $m\leq n$. Then, let us take $q\in L^{\times}$ such that $v_L(q)=p^n$ so that we have $\Phi(A)\cong\bbZ/ p^n\bbZ$. Now, if one choose $q\in L^{\times}$ such that \[q/\pi_L^{p^n}\in \O_K^{\times}((1+\pi_L\O_L)^{p^m}+\pi_K\O_L)\] but \[q/\pi_L^{p^n}\notin\O_K^{\times}((1+\pi_L\O_L)^{p^{m+1}}+\pi_K\O_L)\] then Lemma \ref{order} implies that every element of order $p^m$ lifts into $\A_k(k)$ but no element of order $p^{m+1}$ lifts. To make it possible we need $\O_K^{\times}((1+t\O_L)^{p^{m+1}}+\m\O_L)$ to be a proper subgroup of $\O_K^{\times}((1+t\O_L)^{p^m}+\m\O_L)$. But for a non-negative integer $\ell$ we have \[(1+\pi_L\O_L)^{p^{\ell}}+\pi_K\O_L\subseteq 1+\pi_L^{p^{\ell}}\O_L+\pi_L^d\O_L\] so this is the case if $L/K$ is sufficiently ramified such that $d>p^m$.
\end{rem}

\begin{rem}
Let $A/K$ be an abelian variety with non-Archimedean uniformization \[0\r \Lambda\r G \r A \r 0.\] By \cite[Proposition 6.1 (a)]{LL}, if $A/K$ has split reduction then $G/K$ has split reduction. We can use our construction to show that the converse is not true. The non-Archimedean uniformization of the Tate curve $E/L$ associated to $q\in L^{\times}$ is given by the exact sequence of rigid analytic groups \[0\r q^{\bbZ}\r\bbG_{m,L}\r E\r0.\] Then, the non-Archimedean uniformization of $A=\Res_{L/K}E/K$ is given by \[0\r\Res_{L/K}q^{\bbZ}\r\Res_{L/K}\bbG_{m,L}\r\Res_{L/K}E\r0.\] It follows from Proposition \ref{cpnt} that \[\Phi(\Res_{L/K}\bbG_{m,L})\cong\Phi(\bbG_{m,L}),\] so in particular $\Phi(\Res_{L/K}\bbG_{m,L})$ is free. This implies that $\Res_{L/K}\bbG_{m,L}$ has split reduction but as we saw $A/K$ may not have split reduction.
\end{rem}

\begin{rem}
By \cite[Proposition 1.11]{LL}, if $G/K$ and $G'/K$ are semi-abelian varieties and $f:G\r G'$ is an isogeny of degree prime to $p$ then $G/K$ has split reduction if and only if $G'/K$ has split reduction. We can use our construction to give an explicit example of two isogenous abelian varieties, one which have split reduction and the other one which does not have split reduction. This shows that \cite[Proposition 1.11]{LL} is not true for isogenies of degree divisible by $p$.

Let $p>2$. Let $K=\bbQ_p^{\mathrm{ur}}$ be the maximal unramified extension of the field of $p$-adic numbers and let $L=\bbQ_p^{\mathrm{ur}}(\zeta_p)$ where $\zeta_p$ is a $p$-th root of $1$. We have $d=\left[L:K\right]=p-1$. Let $q=\pi_L^p$ and $q'=\zeta_pq$ and let $E/L$ and $E'/L$ be the Tate curves associated to $q$ and $q'$. These two elliptic curves are isogenous to the Tate curve over $L$ associated to $q^p=q'^p$ by raising to the power $p$ \[\bbG_{m,L}/q^{\bbZ}\stackrel{(\ )^p}{\longrightarrow}\bbG_{m,L}/q^{p\bbZ}\stackrel{(\ )^p}{\longleftarrow}\bbG_{m,L}/q'^{\bbZ}.\] Therefore $E/L$ and $E'/L$ are isogenous. As the Weil restriction of an isogeny is an isogeny, $\Res_{L/K}E/K$ and $\Res_{L/K}E'/K$ are isogenous. But now \[q/\pi_L^p=1\in(1+\pi_L\O_L)^p+\pi_K\O_L\] whereas \[q'/\pi_L^p=\zeta_p\notin(1+\pi_L\O_L)^p+\pi_K\O_L\] because $v_L(\zeta_p-1)=v_L(p)/(p-1)=1$ but $(1+\pi_L\O_L)^p+\pi_K\O_L=1+\pi_L^{p-1}\O_L$. Hence, by Corollary \ref{resE}, $\Res_{L/K}E/K$ has split reduction but $\Res_{L/K}E'/K$ does not have split reduction.
\end{rem}

\section{Splitting properties and the Weil restriction}\label{III}

Our construction in Section \ref{II} leads to the question of the relation between the Weil restriction and splitting properties of semi-abelian varieties. A first answer is given in \cite[Remark 3.10]{LL}. The authors give an example of an elliptic curve which have split reduction but whose Weil restriction does not have split reduction. Using Corollary \ref{resE} we get other examples of this kind since Tate curves have semi-abelian reduction and thus have split reduction. In this section we want to study the general situation.

\subsection{Reduction of Weil restrictions}\label{inffib} Let $K$ be a complete discrete valuation field with ring of integers $\O_K$ and uniformizing element $\pi_K$. Let $G/K$ be a semi-abelian variety with N\'eron model $\G/\O_K$. For any positive integer $n$ we denote by \[\G_n=\G\times_{ \O_K}\O_K/\pi_K^n\O_K\] the infinitesimal fiber, by \[r_n:G(K)\r\G_n(\O_K/\pi_K^n\O_K)\] the reduction map and by $G^n(K)$ the kernel of this reduction map.

\begin{prop}\label{weil}Let $L/K$ be a finite separable extension and let $G/L$ be a semi-abelian variety. We have the following implications :
\begin{enumerate}[(1)]
\item if $\Res_{L/K}G/K$ has split reduction then $G/L$ has split reduction;
\item if $G/L$ has totally not split reduction then $\Res_{L/K}G/K$ has totally not split reduction.
\end{enumerate}
\end{prop}

\begin{proof}
Let us compute $(\Res_{L/K}G)^1(K)$. Using the base change formula (\ref{bc}) we find that this is the kernel of the reduction map \[(\Res_{L/K}G)(K)\r(\Res_{(\O_L/\pi_K\O_L)/k}(\G\times_{\O_L}\O_L/\pi_K\O_L))(k),\] i.e. of the reduction map \[r_d:G(L)\r\G_d(\O_L/\pi_L^d\O_L)\] where $d=\left[L:K\right]$. Hence, we have $(\Res_{L/K}G)^1(K)=G^d(L)\subseteq G^1(L)$.

Let us show assertion (1). Let us suppose that $\Res_{L/K}G/K$ has split reduction. Let $\phi\in\Phi(G)\cong\Phi(\Res_{L/K}G)$ (by Proposition \ref{cpnt}) be an element of finite order $n$. Then $\phi$ lifts to an element $x\in (\Res_{L/K}G)(K)=G(L)$ such that $n\cdot x\in(\Res_{L/K}G)^1(K)\subseteq G^1(L)$. Thus, $G/L$ has split reduction.

Assertion (2) follows from a similar argument.
\end{proof}

In the remaining part of this section we will give a recipe to build counterexamples to the reciprocal of assertions (1) and (2) based on the case by case study in \cite[Section 2]{LL}.

\subsection{Reduction of elliptic curves} Let $L$ be a complete discrete valuation field with valuation $v_L$, ring of integers $\O_L$ and uniformizing element $\pi_L$. Let $E/L$ be an elliptic curve. Let $\E/\O_L$ be its N\'eron model. Assume that $E/L$ is given by a minimal Weierstrass equation \[y^2+a_1xy+a_3x=x^3+a_2x^2+a_4x+a_6,\] with $a_i\in\O_L$ for $i\in\left\{1,2,3,4,6\right\}$. When $a_i\in\pi_L\O_L$ for all $i\in\left\{1,2,3,4,6\right\}$ then the reduced equation has a singular point at $(0,0)$. Let $E^0(L)$ denote the set of rational points in $E(L)$ whose reduction modulo $\pi_L$ is not $(0,0)$. Equivalently $E^0(L)\cong\E^0(\O_L)$ under the isomorphism $E(L)\cong\E(\O_L)$ (see \cite[Corollary IV.9.2]{Si2}) and thus we have \[\Phi(E)\cong E(L)/E^0(L).\] For any positive integer $n$ the subgroup $E^n(L)\subseteq E^0(L)$ defined in \S\ref{inffib} is given by \[E^n(L)=\left\{P\in E(L) \ | \ x(P)/y(P)\in\pi_L^n\O_L\right\}.\] We will denote by $z=-x/y$ the parameter at $\infty$.

\subsection{A key point}\label{ordred} Let us recall an important fact from \cite[Section 2]{LL}. Assume that $E/L$ has additive reduction. Let $P,Q\in E(L)$ be two points whose reductions in $\E_k(k)$ are lying in the same non-trivial component. Then, the orders of those reductions in $\E_k(k)$ are equal. This is only due to the fact that the identity component of the special fiber is killed by $p$ and so it applies more generally in this context. In particular, given a field $K$ as in the introduction such that $L$ is a finite extension of $K$, it applies to $\Res_{L/K}E/K$. Indeed, \[(\Res_{\O_L/\O_K}\E)^0_k\cong\Res_{(\O_L/\pi_K\O_L)/k}\bbG_{a,(\O_L/\pi_K\O_L)}\] by Proposition \ref{cpnt} and the base change formula (\ref{bc}) and thus \[(\Res_{\O_L/\O_K}\E)^0_k(k)\cong\O_L/\pi_K\O_L\] which is killed by $p$. Note that $(\Res_{\O_L/\O_K}\E)^0_k$ is a unipotent algebraic group over $k$.

\subsection{First example} Let us follow \cite[\S 2.13]{LL}. Let $p=3$. Let $E/L$ be of type $\mathbf{IV}$. Then, we have $\Phi(E)\cong\bbZ/3\bbZ$. It is shown that we may assume $a_1=a_3=0$ and that we have the following inequalities for the valuations of the coefficients 
\[\begin{array}{|c|c|c|c|c|c|c|}\hline a_2 & a_4 & a_6 & b_2 & b_4 & b_6 & b_8 \\ \hline \geq 1 & \geq 2 & 2 & \geq 1 & \geq 2 & 2 & \geq 3 \\ \hline \end{array}\] with the relations $b_2=4a_2$, $b_4=2a_4$, $b_6=4a_6$ and $b_8=4a_2a_6-a_4^2$. As we mentioned in \S\ref{ordred}, to study the splitting properties of $E/L$ or $\Res_{L/K}E/K$ it is enough to consider one point with non-trivial image in $\Phi(E)\cong\Phi(\Res_{L/K}E)$. Let $P=(0,y(P))$ with $y(P)^2=a_6$ which is clearly not in $E^0(L)$. Assume that $P$ is not of $3$-torsion. We use the formuli in \cite[III.2.3]{Si1} to compute $v_L(z(3P))$. 

We have \[\begin{array}{l}x(2P)=\displaystyle-\frac{b_8}{b_6},\\y(2P)=\displaystyle-\frac{a_4x(2P)}{2y(P)}-\frac{a_6}{y(P)},\\x(3P)=\displaystyle\left(\frac{y(2P)-y(P)}{x(2P)}\right)^2-a_2-x(2P),\\y(3P)=\displaystyle-\left(\frac{y(2P)-y(P)}{x(2P)}\right)x(3P)-y(P).\end{array}\] Thus,\[x(3P)=\frac{1}{x(2P)^2}\left(\frac{a_4x(2P)+4a_6}{2y(P)}\right)^2-a_2-x(2P)\] and by considering the valuations we get \[v_L(x(3P))=6-2v_L(b_8)\leq0.\] Then, we have \[y(3P)=\frac{1}{x(2P)}\left(\frac{a_4x(2P)+4a_6}{2y(P)}\right)x(3P)-y(P)\] and by considering the valuations we have \[v_L(y(3P))=3-v_L(b_8)+v_L(x(3P)).\] Hence, we find that \[v_L(z(3P))=v_L(b_8)-3\geq0.\] 

By definition $E/L$ does not have split reduction if and only if $3P\notin E^1(L)$. We recover here that this is equivalent to $v_L(b_8)=3$. Set $m=v_L(b_8)-3$. If $m>0$ then $E/L$ has split reduction but for a subextension $L/K$ of degree $d>m$, if it exists, $3P\notin E^d(L)=(\Res_{L/K}E)^1(L)$, i.e. $\Res_{L/K}E$ does not have split reduction. Whence, the reciprocal to Proposition \ref{weil} (1) is false.

\subsection{Second example} Let us follow \cite[\S 2.6]{LL}. Let $p=2$. Let $E/L$ be of type $\mathbf{I}_0^{\ast}$. We have $\Phi(E)\cong\bbZ/2\bbZ\times\bbZ/2\bbZ$. It is shown that we can find three points \[P_i=(\pi_L\alpha_i,0)\in E(L),\] with $\alpha_i\in\O_L$ for $i\in\left\{1,2,3\right\}$ whose reductions modulo $\pi_L$ are distinct, such that their images in $\Phi(E)$ are the three disctinct non-trivial points. More precisely, we have \[x^3+a_2x^2+a_4x+a_6=(x-\pi_L\alpha_1)(x-\pi_L\alpha_2)(x-\pi_L\alpha_3).\] We have the following inequalities for the valuations of the coefficients \[\begin{array}{|c|c|c|c|c|c|c|c|c|}\hline a_1 & a_2 & a_3 & a_4 & a_6 & b_2 & b_4 & b_6 & b_8 \\ \hline \geq 1 & \geq 1 & \geq 2 & \geq 2 & \geq 3 & \geq 2 & \geq 3 & \geq 4 & \geq 4 \\ \hline \end{array}\] We will compute $v_L(z(2P_i))$ for $i\in\left\{1,2,3\right\}$ and $P_i$'s that are not of $2$-torsion.

Let us write $x_i=\pi_L\alpha_i$ for $i\in\left\{1,2,3\right\}$. Without loss of generality we may do the computation taking $i=1$. Recall that we have \[x(2P)=\frac{x^4-b_4x^2-2b_6x-b_8}{4x^3+b_2x^2+2b_4x+b_6}.\] Moreover $x^4-b_4x^2-2b_6x-b_8$ is congruent to $(x^2-a_4)^2$ modulo $\pi^4$ and $a_4=\pi^2(\alpha_1\alpha_2+\alpha_1\alpha_3+\alpha_2\alpha_3)$. Now $x_1^2-a_4=\pi^2(\alpha_1+\alpha_2)(\alpha_1+\alpha_3)$ and thus the valuation of the numerator of $x(2P_1)$ is $4$. As in \cite[\S 2.6]{LL}, the valuation of the denominator of $x(2P_1)$ is $2v_L(a_1x_1+a_3)$ and so we have \[v_L(x(2P_1))=4-2v_L(a_1x_1+a_3)\leq0.\] Then, we have \[y(2P_1)=-(\lambda+a_1) x(2P_1) -\nu - a_3\] where \[\lambda = \frac{3x_1^2+2a_2x_1+a_4}{a_1x_1+a_3}\] and \[\nu=\frac{-x_1^3+a_4x_1+2a_6}{a_1x_1+a_3}.\] Now, we get \begin{multline*}-(a_1x_1+a_3)y(2P_1)=(3x_1^2+2a_2x_1+a_4+a_1(a_1x_1+a_3))x(2P_1)\\-x_1^3+a_4x_1+2a_6+a_3(a_1x_1+a_3).\end{multline*} Using that $3x^2+a_4$ have valuation $4$ we find that \[v_L(y(2P))=2+v_L(x(2P))-v_L(a_1x+a_3).\] Hence, for $i\in\left\{1,2,3\right\}$ we have \[v_L(z(2P_i))=v_L(a_1x_i+a_3)-2\geq0.\] 

Let $m_i=v_L(a_1x_i+a_3)-2$ for $i\in\left\{1,2,3\right\}$. Assume that $E$ does not have totally not split reduction so that $m_i>0$ for some $i$. Then, for a subextension $L/K$ of degree sufficiently large, $\Res_{L/K}E$ has totally not split reduction. Whence, the reciprocal to Proposition \ref{weil} (2) is false.

\section{Split reduction of Jacobian varieties after tamely ramified extensions}\label{IV}

Let $K$ be a complete discrete valuation field with ring of integers $\O_K$ and residue field $k$ of characterisitc exponent $p$. Let $K^s$ be a fixed separable closure of $K$. Let $G/K$ be a semi-abelian variety of dimension $g$. Let $\G/\O_K$ be its N\'eron model with special fiber $\G_k/k$. 

\subsection{Edixhoven's filtration} In the case where $G/K$ is abelian, Edixhoven defined a filtration on $\G_k/k$ by closed subgroups (see \cite{Ed}). It was extended to semi-abelian varieties by Halle and Nicaise in \cite{HN3}. Let us recall the construction of this filtration, following \cite[\S 5.1.3]{HN1}. For every positive integer $d$ prime to $p$, let $K(d)/K$ be the unique degree $d$ extension of $K$ in $K^s$ and let $\O_{K(d)}$ be the ring of integers of $K(d)$. Since $K(d)/K$ is tamely ramified, it is Galois. Let us denote by $\Gamma$ the Galois group $\Gal(K(d)/K)$ and by $\G(d)/\O_{K(d)}$ the N\'eron model of $G_{K(d)}=G\times_{ K} K(d)$. Let \[\W=\Res_{\O_{K(d)}/\O_K}\G(d)\] be the Weil restriction of $\G(d)/\O_{K(d)}$ under the extension $\O_{K(d)}/\O_K$. Then, $\Gamma$ acts on $\W/\O_K$ and the fixed locus $\W^{\Gamma}/\O_K$ is canonically isomorphic to $\G/\O_K$ (see \cite[Theorem 4.2]{Ed} or \cite[Proposition 4.1]{HN3}). Let $\pi_{K(d)}$ be a uniformizing element of $K(d)$. For every $i\in\left\{0,\ldots,d\right\}$, the reduction modulo $\pi_{K(d)}^i$ defines a morphism of group schemes \[(\W^{\Gamma})_k=(\W_k)^{\Gamma}\r \Res_{(\O_{K(d)}/\pi_{K(d)}^i\O_{K(d)})/k}(\G(d)\times_{ \O_{K(d)}} \O_{K(d)}/\pi_{K(d)}^i\O_{K(d)})\] whose kernel is denoted by $F^i_d\G_k$. We get a filtration \[\G_k=F^0_d\G_k\supset F^1_d\G_k\supset \cdots \supset F^d_d\G_k = 0\] on $\G_k/k$ by closed subgroups, and $F^i_d\G_k/k$ is a smooth and connected unipotent algebraic group for all $i>0$. Let us denote by \[\Gr^i_d\G_k=F^i_d\G_k/F^{i+1}_d\G_k\] the graded quotients of this filtration. We say that $j\in\left\{0,\ldots,d-1\right\}$ is a \emph{$K(d)$-jump of $G/K$} if $\dim (\Gr^j_d\G_k) >0$ and we call this dimension the \emph{multiplicity of $j$}.

Edixhoven also introduced a filtration on $\G_k/k$ by rational indices that captures the filtrations introduced above simultaneously for all $d$. For every rational number $\alpha=a/b$ in $\bbZ_{(p)}\cap\left[0,1\right[$, with $a,b$ nonnegative integers and $b$ prime to $p$, we put \[\F^{\alpha}\G_k=F^a_b\G_k.\]  By \cite[Lemma 4.11]{HN3}, this definition does not depend on the choice of $a$ and $b$ and we get a filtration $\F^{\bullet}\G_k$ of $\G_k/k$ by closed subgroups. Note that there are only finitely many closed subgroups occuring in the filtration $\F^{\bullet}\G_k$ because $\G_k/k$ is Noetherian. Let $\rho$ be an element of $\bbR\cap\left[0,1\right[$. Set $\F^{>\rho}\G_k=\F^{\beta}\G_k$, where $\beta$ is any value in $\bbZ_{(p)}\cap\left]\rho,1\right[$ such that $\F^{\beta'}\G_k=\F^{\beta}\G_k$ for all $\beta'\in\bbZ_{(p)}\cap\left]\rho,\beta\right]$. If $\rho\neq0$, set $\F^{<\rho}\G_k=\F^{\gamma}\G_k$, where $\gamma$ is any value in $\bbZ_{(p)}\cap\left[0,\rho\right[$ such that $\F^{\gamma'}\G_k=\F^{\gamma}\G_k$ for all $\gamma'\in\bbZ_{(p)}\cap\left[\gamma,\rho\right[$. Set $\F^{<0}\G_k=\G_k$. Then, let us define \[\Gr^{\rho}\G_k=\F^{<\rho}\G_k/\F^{>\rho}\G_k\] for every $\rho\in\bbR\cap\left[0,1\right[$. We say that $j\in\bbR\cap\left[0,1\right[$ is a \emph{jump of $G$/K} if $\dim(\Gr^j\G_k)>0$ and we call this dimension the \emph{multiplicity of $j$}. Counted with multiplicities, $G/K$ has exactly $g$ jumps.

\subsection{The case of Jacobian varieties}\label{S42} Let $C/K$ be a smooth, proper and geometrically connected curve of genus $g>0$. Assume that $C/K$ has a divisor of degree one. Define the \emph{stabilization index} of $C/K$ (see \cite[Definition 3.2.2.3]{HN1}) as the least common multiple of the multiplicities of the principal components (i.e. components of positive genus or intersecting the other components at at least three points) in the special fiber of the minimal model with strict normal crossings of $C$ over $\O_K$. Let us denote the stabilization index of $C/K$ by $e(C/K)$. Let $L$ be the minimal finite separable extension of $K$ in $K^s$ such that $C\times_K L$ has semistable reduction. If this extension is tamely ramified then $e(C/K)=\left[L:K\right]$ (see \cite[Proposition 3.2.2.4]{HN1}) but this is false in general.

\begin{prop} \label{propjac}
Let $J=\Jac(C)/K$ be the Jacobian variety of $C/K$ and let $\J/\O_K$ be its N\'eron model. Let $U/k$ be the unipotent part of $\J_k^0/k$. Let $d$ be a positive integer prime to $p$. If $d>e(C/K)$ then $U/k$ is the kernel of the canonical morphism \[\J_k\r\J(d)_k.\]
\end{prop}

\begin{proof}
Let $u$ be the dimension of $U$. By \cite[Proposition 5.4.3]{EHN}, $u$ is equal to the number of nonzero jumps of $J/K$ counted with multiplicities. Hence $\F^{>0}\J_k$ is a connected unipotent subgroup of $\J_k$ of dimension $u$. This implies that $\F^{>0}\J_k=U$. By \cite[Corollary 5.3.1.5]{HN1}, the jumps of $J/K$ are rational numbers and the stabilization index $e(C/K)$ is their least common denominator. This implies that the smallest nonzero jump of $J/K$ is $\geq 1/e(C/K)$. Then, $1/d<1/e(C/K)$ implies that $\F^{>0}\J_k=\F^{1/d}\J_k$. Thus, we get \[U=\F^{>0}\J_k=\F^{1/d}\J_k=F^1_d\J_k.\] Now, the fact that $U/k$ is the kernel of the canonical morphism \[\A_k\r\A(d)_k\] follows from the definition of the filtration $F^{\bullet}_d\J_k$ (see \cite[Remark 6.4.6]{Ed}).
\end{proof}

\begin{thm} \label{thm}
Let $J=\Jac(C)/K$ be the Jacobian variety of $C/K$. Let $d$ be a positive integer prime to $p$. If $d>e(C/K)$, then $J_{K(d)}/K(d)$ has split reduction.
\end{thm}

\begin{proof}
First, let $a=\mathrm{gcd}(e(C/K),d)$. Then $K(d)/K(a)$ is a tamely ramified extension of degree $d/a$. Let $C(a)=C\times_{ K}  K(a)$. By \cite[Corollary 3.2.2.10]{HN1}, the stabilization index of $C(a)/K(a)$ is given by \[e(C(a),K(a))=e(C/K)/a.\] Hence, we have $d/a>e(C(a),K(a))$. Let $\J/\O_K$ be N\'eron model of $J/K$. The previous proposition applied to $C(a)/K(a)$ implies that the unipotent part $U/k$ of $\J(a)_k^0/k$ is the kernel of the canonical morphism \[\J(a)_k\r\J(d)_k.\]

Now, by \cite[Proposition 1.3.3.1]{HN1}, we have $|\Phi(J_{K(d)})|=(d/a)^{t_{K(a)}}|\Phi(J_{K(a)})|$, where $t_{K(a)}$ is the dimension of the toric part of $\J(a)_k^0$. As $d/a$ is prime to $p$, the orders of $\Phi(J_{K(a)})_p$ and $\Phi(J_{K(d)})_p$ are the same. Moreover, by \cite[Theorem 1]{ELL}, the kernel of the canonical morphism \[\Phi(J_{K(a)})\r\Phi(J_{K(d)})\] is killed by $d/a$ which is prime to $p$. Hence, this morphism is an isomorphism on the $p$-primary parts of these groups.

Finally, we can proceed as in the proof of \cite[Proposition 3.3]{LL}. Let $\phi'\in \Phi(J_{K(d)})$ be an element of order $p^r$ where $r>0$. Let $\phi\in\Phi(J_{K(a)})$ be the corresponding element under the above morphism and let $x\in\J(a)_k$ be in the preimage of $\phi$. Let $x'$ be the image of $x$ in $\J(d)_k$. As in the proof of \cite[Propositin 1.4 (b)]{LL}, we may assume that $p^r\cdot x\in U$. Thus, we have $p^r\cdot x'=0$. Since $x'$ is in the preimage of $\phi'$, this proves that $J_{K(d)}/K(d)$ has split reduction.
\end{proof}

\begin{rem}
In Proposition \ref{propjac} and Theorem \ref{thm} we need assume that the abelian variety is the Jacobian variety of a curve of index one to be able to use \cite[Proposition 5.4.3]{EHN}, \cite[Corollary 5.3.1.5]{HN1} and \cite[Proposition 1.3.3.1]{HN1}. If one can prove these results without this hypothesis then Theorem \ref{thm} would be true for arbitrary semi-abelian varieties over $K$ provided that we have a suitable definition of the stabilization index $e(G/K)$ of arbitrary semi-abelian varieties $G/K$. This problem is discussed in \cite[Part 4, \S 1]{HN1}.
\end{rem}

\begin{cor}\label{yes611}
Let $J=\Jac(C)/K$ be the Jacobian variety of $C/K$. There exists a constant $c$ depending on $g$ only such that if $M/K$ is any tamely ramified extension of degree $> c$, then $J_M/M$ has split reduction. 
\end{cor}

\begin{proof}
All we need to prove is that if $g>0$ is fixed, then the stabilization index of a curve $C/K$ of genus $g$ is bounded by a constant $c$. This follows from \cite[Corollary 1.7]{AW} for $g\geq2$ and from Kodaira-N\'eron classification for elliptic curves.
\end{proof}

\begin{rem}
In the case of elliptic curves, one checks that the stabilization index is at most $6$. Hence, we almost recover Theorem \ref{21} (3) which states that $M/K$ of degree $\geq4$ is enough to acquire split reduction.
\end{rem}

Finally, let us remark that in the case of elliptic curves we have the following alternative result to Theorem \ref{thm}.

\begin{prop}
Let $E/K$ be an elliptic curve. Let $L/K$ be the minimal finite separable extension such that $E_L/L$ has semi-abelian reduction. Let $d$ be a positive integer prime to $p$. If $d>\left[L:K\right]$, then $E_{K(d)}/K(d)$ has split reduction.
\end{prop}

\begin{proof}
By Theorem \ref{21} (c), if $d\geq 4$ then $E_{K(d)}/K(d)$ has split reduction. Hence, the only case to tackle is then $\left[L:K\right]=2$ and $d=3$. In this case, the reduction type cannot be $\mathbf{IV}$ or $\mathbf{IV}^{\ast}$ because for these types the groups of components are of order $3$ and should be killed by $\left[L:K\right]$ by \cite[Theorem 1]{ELL}. If $E/K$ has reduction type $\mathbf{I}_n^{\ast}$ then $e(E/K)=2$ and the result follows from Theorem \ref{thm}. Thus, we may assume that the reduction type of $E/K$ is either $\mathbf{II}$, $\mathbf{III}$, $\mathbf{III}^{\ast}$ or $\mathbf{II}^{\ast}$. Recall that we have the following formula for Swan conductors \[\delta(E_{K(d)}/K(d))=d\cdot\delta(E/K).\] By \cite[Lemma 3.4]{LL}, the reduction type of $E_{K(d)}/K(d)$ will be either $\mathbf{II}$, $\mathbf{III}$, $\mathbf{III}^{\ast}$, $\mathbf{II}^{\ast}$ or $\mathbf{I}_0^{\ast}$. By Theorem \ref{21} (a), in order that $E_{K(d)}/K(d)$ does not have split reduction we must have $\delta(E/K)=1$. Let $\Delta$ be the minimal discriminant of $E/K$. Then, Ogg's formula implies that $v_K(\Delta)\in\left\{3,4,10,11\right\}$. But this is not possible if $\left[L:K\right]=2$. Hence, $E_{K(d)}/K(d)$ has split reduction.
\end{proof}

\section{Splitting properties and the Swan conductor}\label{V}

\subsection{Swan conductor of Weil restrictions} Let $L/K$ be a finite separable extension. Let $E/L$ be an elliptic curve and let $A=\Res_{L/K}E/K$ be its Weil restriction under the extension $L/K$. By \cite[\S 1, Lemma]{Mi}, the Swan conductor of $A/K$ is given by the following formula \begin{equation}\delta(A/K)=\delta(E/L)+2(v_L(\D_{L/K})-(p-1)),\label{swan}\end{equation} where $\D_{L/K}$ is the different ideal of the extension $L/K$. Before tackling Question \ref{69} let us start with some preliminary results. 
 
\subsection{Some preliminaries}\label{S52} Let $L$ be a complete discrete valuation field with algebraically closed residue field. Let $E/L$ be an elliptic curve which acquires multiplicative reduction over a quadratic separable extension $M/L$. Then, there exists $q\in L^{\times}$ such that $E/L$ is isomorphic over $M$ to the Tate curve defined by $q$ (see \cite[Theorem V.5.3]{Si2}). Hence, the non-Archimedean uniformization of $E/L$ is given by the following exact sequence \[0\r\Lambda\r\Res^1_{M/L}\bbG_{m,M}\r E\r0,\] where $\Res^1_{M/L}\bbG_{m,M}$ is the norm torus defined in \S\ref{S110} and $\Lambda/L$ is the lattice \[\Res_{M/L}q^{\bbZ}\cap\Res^1_{M/L}\bbG_{m,M},\] where $q^{\bbZ}/M$ is the constant lattice in $\bbG_{m,M}/M$.

\begin{lem}\label{unif}
Keep the previous notation. Assume that $v_L(q)$ is odd. Then, the following assertions are equivalent :
\begin{enumerate}
\item $E/L$ has totally not split reduction;
\item $E/L$ does not have split reduction;
\item $\Res^1_{M/L}\bbG_{m,M}/L$ has totally not split reduction;
\item $\Res^1_{M/L}\bbG_{m,M}/L$ does not have split reduction.
\end{enumerate}
\end{lem}

\begin{proof}
As $v_L(q)$ is odd, the group of components $\Phi(E)$ is cyclic (see \cite[Theorem 2.8]{Lor}). Thus, (1) $\Leftrightarrow$ (2) follows from \cite[Proposition 1.4(e)]{LL}. Then, (3) $\Leftrightarrow$ (4) is trivial since $\Phi(\Res^1_{M/L}\bbG_{m,M})\cong\bbZ/2\bbZ$ by \cite[Proposition 4.17 (2)]{LL}. Now, the lemma follows from by \cite[Proposition 6.1 (a) and (b)]{LL}.
\end{proof}

\begin{prop}\label{isog}
Let $L/K$ be a quadratic separable extension and let $\sigma$ be the non-trivial element of $\Gal(L/K)$. Let $E/L$ be an elliptic curve which acquires multiplicative reduction over a quadratic separable extension $M/L$. Let $q\in L^{\times}$ be such that $E_M/M$ is isomorphic to the Tate curve defined by $q$. If $q/\sigma(q)$ is not a root of $1$ then the abelian variety $A=\Res_{L/K}E/K$ is simple.
\end{prop}

\begin{proof}
We have $\dim(A)=2$ (see \S\ref{S22}). Let $B/K$ be an abelian variety of dimension $1$ and assume that there is a non-trivial $K$-morphism \[B\r A.\] Then, by property of the Weil restriction we get an $L$-isogeny \[B\times_K L\r E.\] From this, we get another $L$-isogeny \[B\times_K L\r\ ^{\sigma}E.\] Hence, $E$ and $^{\sigma}E$ should be $L$-isogenous. Now, the elliptic curve $ ^{\sigma}E/L$ is isomorphic over $M$ to the Tate curve defined by $\sigma(q)\in L^{\times}$. If $q/\sigma(q)$ is not a root of $1$, then $q$ and $\sigma(q)$ are not commensurable, i.e. there are no non-zero integers $m,n$ such that $q^n=\sigma(q)^m$. By \cite[\S Isogenies, Theorem]{Ta}, this implies that $E$ and $ ^{\sigma}E$ are not isogenous over $M$ hence neither over $L$. Whence, $A/K$ is simple.
\end{proof}

\subsection{Counterexample to Question \ref{69}}\label{no69}
Assume that $K$ is of residue characteristic $2$. Let $L/K$ be a quadratic separable extension. It follows from \cite[Lemmata 4.1(b) and 4.5]{LL} and Theorem \ref{46} (1) that one can choose a quadratic separable extension $M/L$ such that $\Res^1_{M/L}\bbG_{m,M}$ does not have split reduction. Let $\sigma$ be the non-trivial element of $\Gal(L/K)$. Let us choose $q\in L^{\times}$ with $v_L(q)$ odd and such that $q/\sigma(q)$ is not a root of $1$. Consider the elliptic curve $E/L$ given by the uniformization in \S\ref{S52}. It follows from Lemma \ref{unif} that $E/L$ has totally not split reduction. Then, by Proposition \ref{weil}, $A=\Res_{L/K}E/K$ has totally not split reduction. Finally, it follows from Proposition \ref{isog} that $A/K$ is simple. We already mentionned in \S\ref{ordred} that \[\A^0_k=\Res_{(\O_L/\pi_K\O_L)/k}\bbG_{a,(\O_L/\pi_K\O_L)}\] is an unipotent algebraic group over $k$ and thus the toric rank of $A/K$ is $0$. Now we have \[1\leq\delta(E/L)\leq3\] by Theorem \ref{21} (1) and thus it is clear from Formula (\ref{swan}) that $\delta(A/K)$ really depends on the field $K$.
Whence, the answer to Question \ref{69} is negative even for simple abelian varieties over $K$.

To conclude, let us give a concrete example. Let $\bbQ_2^{ur}$ be the maximal unramified extension of the field of $2$-adic numbers. Let $d\geq2$ be an integer. Let $K=\bbQ_2^{ur}(\pi_K)$ with $\pi_K^d=2$ and let $L=K(\pi_L)$ with $\pi_L^2=\pi_K$. The different ideal of $L/K$ is $2\pi_L\O_L$ (use \cite[Corollaire III.6.2]{Se1}). Let $q=\pi_L^n(1+\pi_L)$ with $n$ odd. Let $\sigma$ be the non-trivial element of $\Gal(L/K)$. We have \[q/\sigma(q)=-(1+\pi_L)/(1-\pi_L),\] and this is not a root of $1$ because it is not in $\bbQ_2^{ur}$ (apply $\sigma$). Now, let $M/L$ be the extension defined by the Eisenstein polynomial $t^2+\pi_L t+\pi_L$. The different ideal of the extension $M/L$ is $\D_{M/L}=\pi_L\O_M$ (use \cite[Corollaire III.6.2]{Se1}). Then, by \cite[Lemmata A.1 (b) and 4.5]{LL} we have \[\begin{array}{lll}\delta(\Res^1_{M/L}\bbG_{m,M}/L)&=&v_M(\D_{M/L})-(2-1)\\&=&1.\end{array}\] It follows from Theorem \ref{46} (1) that $\Res^1_{M/L}\bbG_{m,M}$ has totally not split reduction. Now, by \cite[Proposition 6.4 (b)]{LL}, we have \[\begin{array}{lll}\delta(E/L)&=&2\delta(\Res^1_{M/L}\bbG_{m,M}/L)\\&=&2.\end{array}\] Finally, by formula (\ref{swan}), we have \[\begin{array}{lll}\delta(A/K)&=&\delta(E/L)+2(v_L(2\pi_L)-(2-1))\\&=&2+4d.\end{array}\]

\begin{rem}
As suggested in \cite{LL}, one may ask the same question for the more restricted class of abelian varieties such that the representation of the absolute Galois group $\Gal(K^s/K)$ on the Tate module $T_{\ell}(A)$, $\ell\neq p$, is irreducible. We do not know any counterexample of this kind.
\end{rem}

\subsection{Brumer and Kramer's bound} Before tackling Question \ref{610}, let us recall the bound for the Swan conductor given in \cite{BK}. For any abelian variety $A/K$ and any extension $M/K$, let us denote by $a_M$ and $t_M$ the dimensions of the abelian and toric parts of the reduction of $A_M/M$ respectively. Let us consider the following function on integers \[\lambda_p(n)=\sum_{i=0}^sir_ip^i,\] where \[n=\sum_{i=0}^sr_ip^i\] is the $p$-adic expansion of $n$, with $0\leq r_i<p$.

\begin{prop}\cite[Proposition 3.11]{BK}\label{BK}

Assume that $K$ is of characteristic $0$. Let $A/K$ be an abelian variety which acquires semi-abelian reduction over $L$. Let $K_1$ be the subfield of $L$ fixed by the first ramification subgroup of $\Gal(L/K)$. We have \[\delta(A/K)\leq2(d_t+d_a)pv_K(p)+(p-1)(2\lambda_p(d_t)+\lambda_p(2d_a))v_K(p),\] where $d_t$ and $d_a$ are defined by $t_L-t_{K_1}=(p-1)d_t$ and $a_L-a_{K_1}=(p-1)d_a$ respectively. 
\end{prop}

\subsection{Counterexample to Question \ref{610}}\label{no610} Assume that $K$ is of characteristic $0$ and that it contains the $p$-th roots of $1$. Let $L/K$ be the Kummer extension given by the polynomial $t^p-\pi_K$. Let $E/L$ be an elliptic curve and let $A=\Res_{L/K}E/K$ be its Weil restriction under the extension $L/K$. Using \cite[Corollaire III.6.2]{Se1} we can compute $\D_{L/K}=p\pi_L^{p-1}\O_L$ and then formula (\ref{swan}) gives \[\begin{array}{lll}\delta(A/K)&=&\delta(E/L)+2v_L(p)\\&=&\delta(E/L)+2pv_K(p).\end{array}\]

Assume that $E/L$ is a Tate curve. Then $E/L$ has semi-abelian reduction and thus $\delta(E/L)=0$ so that $\delta(A/K)=2pv_K(p)$. Now let us compute the bound for the Swan conductor from Proposition \ref{BK}. Here, the Galois group $\Gal(L/K)$ coincides with the first ramification subgroup, hence the subfield of $L$ fixed by the latter is simply $K$. By Proposition \ref{mult}, $A_L/L$ has purely multiplicative reduction. Thus, we have $a_K=a_L=0$ hence $d_a=0$ and $t_K=1$ (see the exact sequence (\ref{untor})) whereas $t_L=p$ hence $d_t=1$. We also have $\lambda_p(0)=\lambda_p(1)=0$ and therefore we get \[\delta(A/K)\leq2pv_K(p).\] This bound is exactly the one we achieved. Now we saw in Proposition \ref{resE} that we may choose our elliptic curve $E/L$ such that $A/K$ does not have split reduction and therefore the answer to Question \ref{610} is no in general. 

\begin{rem}
Our example has positive toric rank ($t_K=1$). We do not know any example of abelian variety with toric rank $0$ which does not have split reduction and whose conductor achieves the bound from \cite{BK}.
\end{rem}

\begin{rem} Assume that $K$ is of characteristic $p>0$. Let $L/K$ be the extension given by the Eisenstein polynomial $t^p+a_{p-1}t+a_p$ where $a_{p-1}\in \pi_K\O_K$ is different from $0$ and $v_K(a_p)=1$. This polynomial is separable and the different ideal of the extension $L/K$ is $\D_{L/K}=a_{p-1}\O_L$ (use \cite[Corollaire III.6.2]{Se1}). Now, let $E/L$ be an elliptic curve with totally not split reduction. Let $A=\Res_{L/K}E$ be its Weil restriction under the extension $L/K$. Then, $A/K$ has totally not split reduction by Proposition \ref{weil}. Moreover, by formula (\ref{swan}) we have \[\delta(A/K)=\delta(E/L)+2pv_K(a_{p-1})-(p-1).\] Hence, considering an unbounded family of $a_p$'s we get a family of abelian varieties which have totally not split reduction but unbounded Swan conductors.
\end{rem}

\end{document}